\documentclass[12pt]{amsart}
\usepackage[foot]{amsaddr}
\usepackage{graphicx}
\usepackage{amssymb}
\usepackage{epstopdf}
\usepackage{amsmath,amscd}
\usepackage{amsthm}
\usepackage{url,verbatim}
\usepackage{caption,subcaption}

\usepackage[dvipsnames]{xcolor}

\usepackage{tikz}
\usetikzlibrary{decorations.markings}

\RequirePackage[colorlinks,citecolor=blue,urlcolor=blue]{hyperref}
\usepackage{breakurl}
\theoremstyle{plain}
\newtheorem{theorem}{Theorem}

\newtheorem{proposition}[theorem]{Proposition}

\newtheorem{conjecture}[theorem]{Conjecture}
\newtheorem{example}[theorem]{Example}

\newtheorem{remark}[theorem]{Remark}

\newcommand{\grgr}[1]{{\textcolor{blue}
{[GRG: #1\marginpar{\textcolor{blue}{\bf \hfil $\bigstar$}}]}}}

\newcommand{\mc}[1]{\mathcal{#1}}
\newcommand\es{\varnothing}

\newcommand\ol{\overline}

\newcommand\EE{{\mathbb E}}

\newcommand\HH{{\mathbb H}}
\newcommand\RR{{\mathbb R}}
\newcommand\ZZ{{\mathbb Z}}

\newcommand\PP{{\mathbb P}}
\newcommand\pcH{\pc^{\HH}}

\newcommand{\pp}{\mathrm{P}_p}

\newcommand\om{\omega}

\newcommand\eps{\epsilon}

\newcommand\qq{\qquad}
\newcommand\lam{\lambda}
\newcommand\lc{\lam_{\text{\rm c}}}
\newcommand\olc{\ol\lam_{\text{\rm c}}}
\newcommand\q{\quad}
\newcommand\resp{respectively}

\newcommand\oo{\infty}

\newcommand\sD{{\mathcal D}}

\newcommand\sT{{\mathcal T}}

\newcommand\sS{{\mathcal S}}

\newcommand\la{\lambda}
\newcommand\La{\Lambda}
\newcommand\Lazz{\La_{o}}
\newcommand\Laoz{\La_{e_1}}
\newcommand\Ppp{\PP_{p,\pi}}

\newcommand\pcsite{\pc^{\text{\rm site}}}
\newcommand\pcbond{\pc^{\text{\rm bond}}}

\newcommand\Om{\Omega}

\newcommand\id{{\bf 1}}

\newcommand\pc{p_{\text{\rm c}}}

\newcommand\modelone{one-choice model}
\newcommand\modeltwo{independent model}

\newcommand\pd{\partial}

\newcommand\PPpm{\PP_{p,\mu}}
\newcommand\PPpl{\PP_{p,\la}}

\newcounter{mycount1}\newcounter{mycount2}\newcounter{mycount3}\newcounter{mycount}
\newenvironment{romlist}{\begin{list}{\rm(\roman{mycount1})}%
   {\usecounter{mycount1}\labelwidth=1cm\itemsep 0pt}}{\end{list}}

\newenvironment{letlist}{\begin{list}{\rm(\alph{mycount3})}%
   {\usecounter{mycount3}\labelwidth=1cm\itemsep 0pt}}{\end{list}}
\newenvironment{Alist}{\begin{list}{\rm\MakeUppercase{\alph{mycount}}.}%
   {\usecounter{mycount}\labelwidth=1cm\itemsep 0pt}}{\end{list}}

\newcommand{\M}[1]{#1}

\numberwithin{equation}{section}
\numberwithin{theorem}{section}
\numberwithin{figure}{section}

\DeclareMathOperator{\Binomial}{binom}
\DeclareMathOperator{\Geometric}{geom}
\DeclareMathOperator{\Corr}{corr}

\title{Alignment percolation}
\author[N. R. Beaton, G. R. Grimmett, M. Holmes]{Nicholas R. Beaton$^1$}

\author{Geoffrey R.\ Grimmett$^{1,2}$}

\address{$^1$School of Mathematics \&\ Statistics, The University of Melbourne, 
Parkville, VIC 3010, Australia}
\address{$^2$ Centre for Mathematical Sciences, University of Cambridge, CB2 0WB, United Kingdom} 

\email{nrbeaton@unimelb.edu.au}
\email{grg@statslab.cam.ac.uk}
\email{holmes.m@unimelb.edu.au}
\author{Mark Holmes$^1$}

\begin{document}

\begin{abstract}
The existence (or not) of infinite clusters is explored for two stochastic models
of intersecting line segments in $d \ge 2$ dimensions. Salient features of the phase diagram are established in each case.
The models are based on site percolation on $\ZZ^d$ with parameter $p\in (0,1]$.  
For each occupied site $v$, and for each of the $2d$ possible coordinate directions, 
declare the entire line segment from $v$ to the next occupied site in the given direction to be 
either blue or not blue according to a given 
stochastic rule.  In the \lq\modelone', each occupied site declares one of its $2d$  incident segments to be blue.
In the \lq\modeltwo', the states of different line segments are independent.  
\end{abstract}

\date{19 August 2019, revised 2 July 2020} 

\keywords{Percolation, \modelone, \modeltwo}
\subjclass[2010]{60K35}
\maketitle

\section{The two models}\label{sec:int}
Percolation theory is concerned with the existence of infinite clusters in stochastic geometric models.
In the classical percolation model, sites (or bonds) of $\ZZ^d$ are declared \emph{occupied} with probability
$p$, independently of one another, and the basic question is to 
understand the geometry of the occupied subgraph for different ranges of values of $p$
(see \cite{G99} for a general account of percolation). A number of models of 
\emph{dependent}  percolation have been studied
in various contexts including statistical physics, social networks, and stochastic geometry.
One area of study of mathematical interest has been the geometry of
disks and line segments arising in Poisson processes in $\RR^d$; see, for example,
\cite{del, HM,H16}. Motivated in part by such Poissonian systems, we study here two percolation systems of 
random line segments on the hypercubic lattice $\ZZ^d$ with $d \ge 2$. Each is based on
site percolation on $\ZZ^d$.

These two processes are called the \lq\modelone' and the \lq\modeltwo', \resp, 
with the difference lying in 
the manner in which line segments within the site percolation model are declared to be active (or \lq blue').
In the  \modelone, there is dependence between neighbouring segments, whereas they are
conditionally independent in the \modeltwo. We describe the two models next.

Let $\Om_V=\{0,1\}^{\ZZ^d}$ be the state space of site percolation on
the hypercubic lattice $\ZZ^d$ with $d \ge 2$, 
let $\mc{F}_V$ denote the $\sigma$-field generated by the cylinder sets 
(i.e.~the states of finitely many sites),
and let $\pp$ be product measure on $(\Om_V,\mc{F}_V)$ with 
density $p\in(0,1]$. 
For $\om=(\om_v:v \in \ZZ^d)\in \Om_V$, a vertex $v\in\ZZ^d$ is called \emph{occupied} if
$\om_v=1$, and \emph{unoccupied} otherwise, and we 
write $\eta(\om)=\{v\in\ZZ^d: \om_v=1\}$
for the set of occupied vertices. 
We construct a random 
subgraph of $\ZZ^d$ as follows. First, let $\om\in\Om_V$ be sampled according
to $\pp$. 
An unordered pair of distinct vertices $v_1,v_2\in\eta(\om)$ is called \emph{feasible} if
\begin{romlist}
\item  $v_1$ and $v_2$ differ in exactly one coordinate, 
and 
\item  on the straight line-segment of $\ZZ^d$ joining $v_1$ and $v_2$,
$v_1$ and $v_2$ are the only occupied vertices. 
\end{romlist}
Each feasible pair may be considered as the straight 
line-segment (or, simply, \emph{segment}) of $\ZZ^d$ joining them.

Let $F(\om)$ be the set of all 
feasible pairs of occupied vertices of $\om$, and let  $F_v(\omega)$ denote the set of feasible pairs containing $v$.   
We will declare some subset $S(\omega) \subseteq F(\omega)$ to be \emph{blue}. 
An edge $e$ in the edge set $\EE^d$ of $\ZZ^d$ is called 
\emph{blue} if it lies in a blue segment, and a site is \emph{blue} if it is incident to one or more blue edges.  
We are interested in whether or not there exists an infinite cluster of blue edges.  The relevant state space is then $\Om_E=\{\text{blue, not blue}\}^{\EE^d}$, and $\mc{F}_E$ denotes the $\sigma$-field generated by the cylinder sets.

Roughly speaking, the parameter $p$ controls how much the set $F$ differs from the original lattice.  When $p$ is large, most sites are occupied, and many feasible pairs are simply edges of $\ZZ^d$; indeed, one obtains the entire lattice in the case
$p=1$.  When $p$ is small, segments tend to be long and have many other segments crossing them.

We will work with two specific stochastic constructions of the blue segments.  
We write $o=(0,0,\dots,0)$ for the origin of $\ZZ^d$, and $e_i$ for the unit vector
	in the direction of increasing $i$th coordinate. An edge with endvertices $u$, $v$ is
	written $(u,v)$.  Our first model is called the \modelone.

\begin{example}[The \modelone]
	\label{ex:compass}
	Let $\om\in\Om_V$.   
	Independently for each $v \in \eta(\om)$, choose a segment (say $f_v$) uniformly at random from $F_v(\om)$ 
	and declare it to be \emph{green}; in this case, we say that \emph{$v$ has declared $f_v$ to be green} and that the direction from $v$ along $f_v$ is the \emph{chosen direction} of $v$.  
	A segment $e$ with endvertices $u$, $v$ is declared
	\emph{blue} if and only if at least one of $u$ or $v$ has declared it green.
\end{example}	
	Figure \ref{fig:compass} shows simulations of the set of blue edges of $\ZZ^2$ in the \modelone\ with $p=0.8$ and $p=0.4$.
	\begin{figure}
	\centering
	\begin{subfigure}{0.48\textwidth}
	\includegraphics[width=\textwidth]{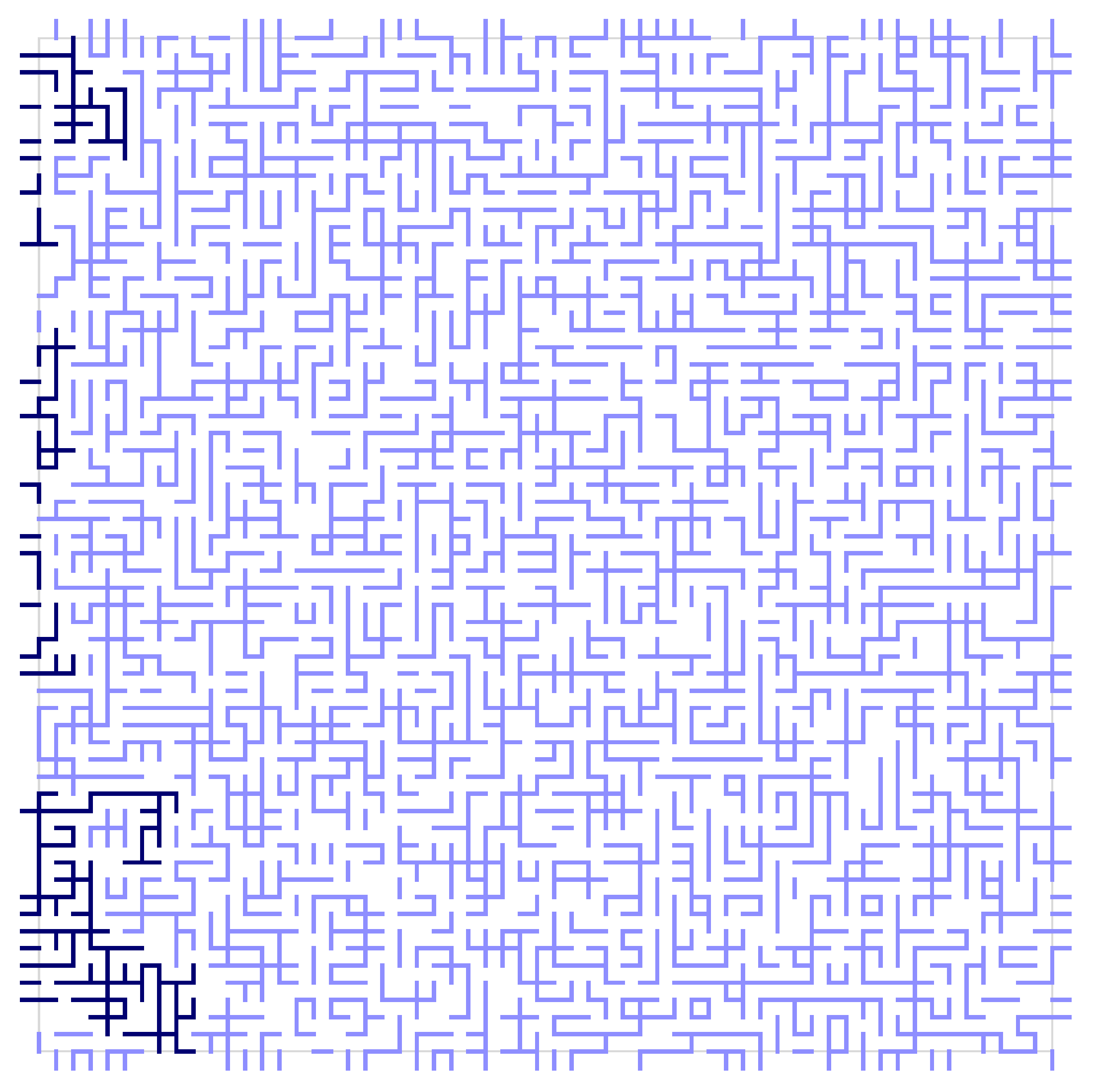}
	\end{subfigure}
	\hspace{0.2cm}
	\begin{subfigure}{0.48\textwidth}
	\includegraphics[width=\textwidth]{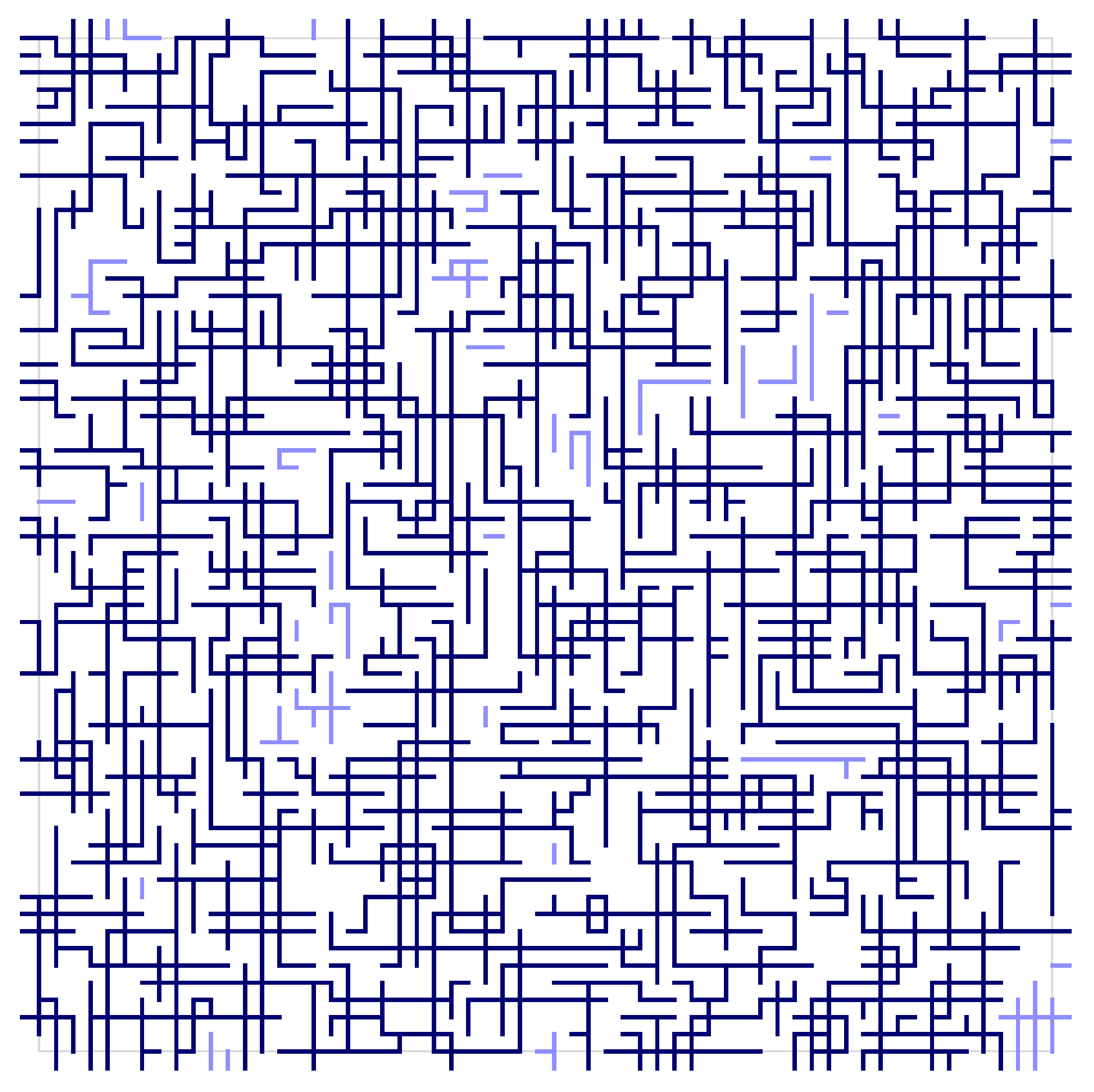}
	\end{subfigure}

		\caption{A $60\times60$ section of $\ZZ^2$ with blue edges in the \modelone\
		of  Example~\ref{ex:compass} indicated, with $p=0.8$ (left) and $p=0.4$ (right). Segments which cross the boundary of the box have been cut off for clarity. Segments which cross the left boundary, and those connected to them by blue edges, are coloured darker.}
		\label{fig:compass}
	\end{figure}

There is no apparent stochastic monotonicity in the set of blue sites as a function of $p$ for the one-choice model.  Elementary calculations show that any given edge $e$ of $\ZZ^d$ is blue with probability 
\begin{equation}
\la:= 1-\left(1-\frac{1}{2d}\right)^2,\label{lambda_def_choice}
\end{equation}
independently of $p$.  The probability that a given vertex is incident to some blue edge of $\ZZ^d$ is $p +(1-p)\bigl\{1-(1-\lambda)^d\bigr\}$, which is increasing in $p$.  On the other hand, the probability that the edges $(o,e_1)$ and $(o,-e_1)$ are both blue (where $o$ is the origin) is $1-(1+p/d)(1-\lam)$, which is decreasing in $p$.  In other words, while the frequency of blue edges does not change with $p$, blue edges tend to be more aligned as we decrease $p$.  Similarly the probability that $(o,e_1)$ and $(o,e_2)$ are both blue is $\lam^2-p(2d-1)^2/(2d)^4$, which is decreasing in $p$.

We refer to the second model as the \modeltwo.
\begin{example}[The \modeltwo]\label{ex:indep}
Let $\om\in\Om_V$ and $\la\in[0,1]$.
Independently for each segment in $F(\om)$, declare the segment to be blue with probability $\lambda$.	
	\end{example}	
\begin{figure}
\centering
	\begin{subfigure}{0.48\textwidth}
	\includegraphics[width=\textwidth]{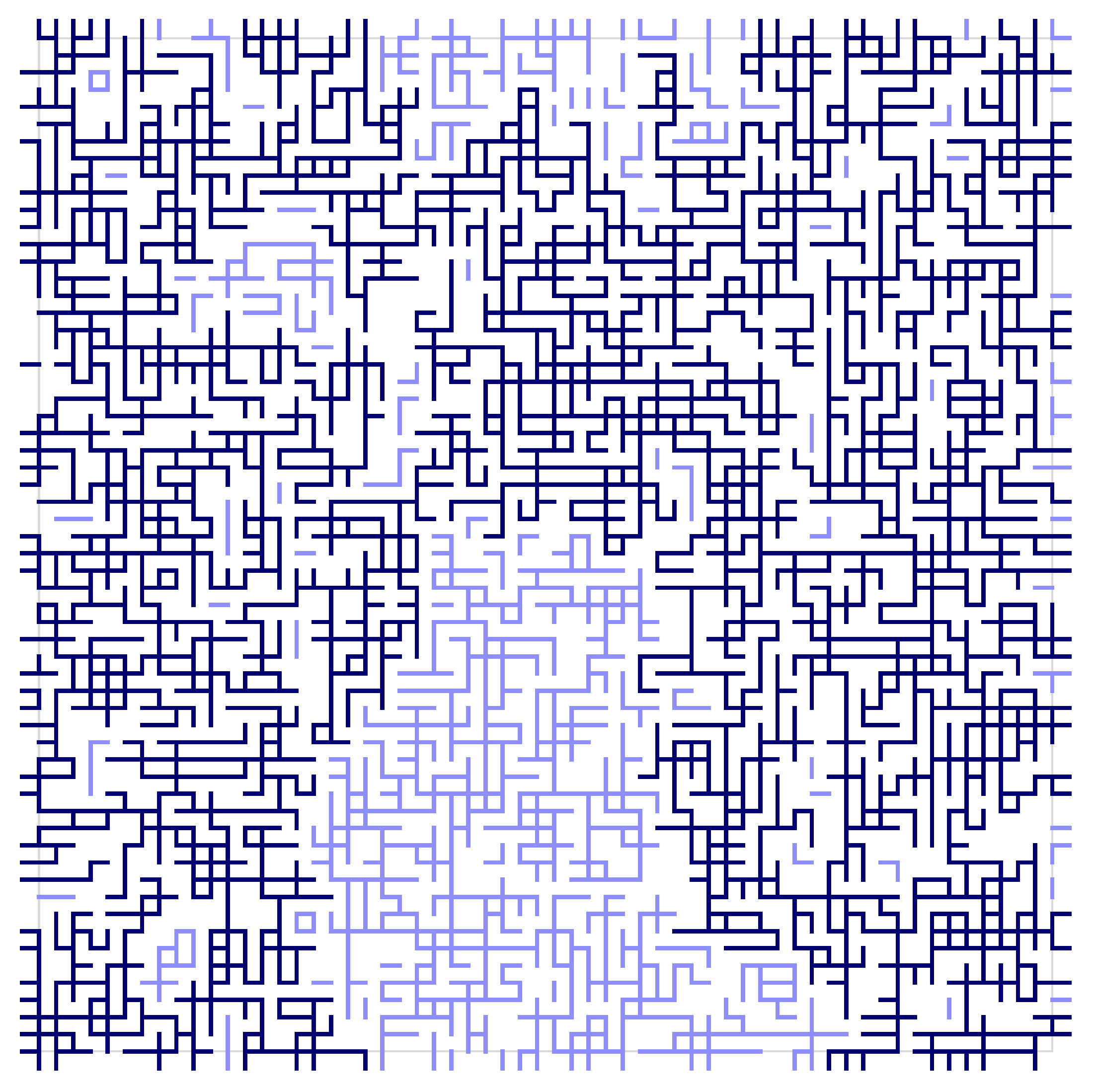}
	\end{subfigure}
	\hspace{0.2cm}
	\begin{subfigure}{0.48\textwidth}
	\includegraphics[width=\textwidth]{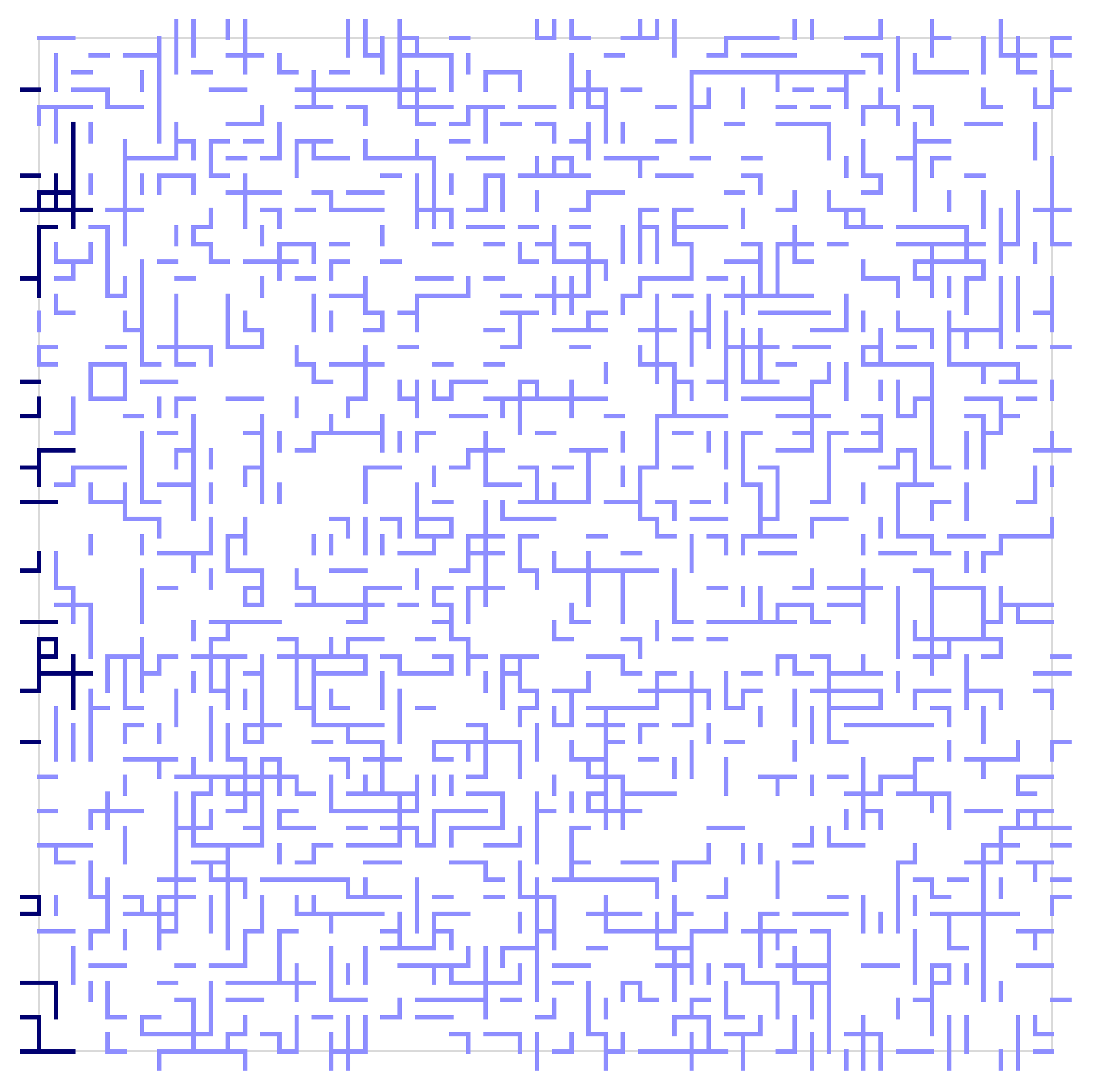}
	\end{subfigure}
	\caption{A $60\times60$ section of $\ZZ^2$ with blue edges in the \modeltwo\
	of  Example~\ref{ex:indep} indicated, with $p=0.8$ and with $\lambda=0.5$ (left) and $\lambda=0.3$ (right). 
	Segments which cross the boundary of the box have been cut off for clarity. 
	Segments which cross the left boundary, and those connected to them, are coloured darker. 
	Occupied sites that are not blue have been omitted.}
	\label{fig:indep}
\end{figure}
Figure \ref{fig:indep} shows simulations of the set of blue edges of $\ZZ^2$ in the  \modeltwo\  
	with $p=0.8$, and with $\lambda=0.5$ and $\lambda=0.3$.

There is no apparent stochastic monotonicity in the set of blue sites as a function of $p$ (with fixed $\lambda$) for the independent model.  An easy calculation shows that the probability that a given edge is blue is $\lambda$, so in particular it does not depend on $p$.
On the other hand, the 
 probability that a given vertex is incident to a blue edge is 
 $1- (1-\lambda)^d+p\left\{(1-\lambda)^d-(1-\lambda)^{2d}\right\}$, which is increasing in $p$.  For two distinct edges $e$ and $e'$ that are collinear, one can show that  $\Corr(\id_{\{e\text{\rm\ is blue}\}},\id_{\{e'\text{\rm\ is blue}\}}) = (1-p)^k$, where $\id_A$ denotes the indicator function of the event $A$, and $k\geq1$ is the number of lattice sites between $e$ and $e'$. 	Thus correlations are decreasing in $p$ (if $e$ and $e'$ are not collinear then  the events $\{e\text{\rm\ is blue}\}$ 		
 and $\{e'\text{\rm\ is blue}\}$ are independent).
	
In each example above, there are two levels of randomness. First, one chooses a site percolation configuration $\om$
according to the product measure $\pp$, and then the blue edges are selected 
according to an appropriate conditional measure
$P_\om$. 
  We formalize this as follows.
Let $\PP_p$ denote the probability measure on the space $(\Om_V\times \Om_E,\sigma(\mc{F}_V\times \mc{F}_E))$, such that, for measurable rectangles of the form $A\times B$,
\begin{equation}
\label{joint}
\PP_p(A \times B)=\int_A P_\omega(B)\, d\pp(\omega).
\qq A\in \mc{F}_V,\ B\in \mc{F}_E,
\end{equation}
Whenever we need to make the dependence on $\lambda$ explicit,
as  in Example \ref{ex:indep}, this will be written as $\PP_{p,\lambda}(A\times B)=\int_A P^\lambda_\omega(B)\, d\pp(\omega)$, where $P^\lambda_\omega$ is the (quenched) law of the independent model with parameter $\lambda$, 
conditional on the site percolation configuration $\omega$.  We will sometimes abuse notation and use the notation $\PP_p$ to denote the `annealed' probability measure on $(\Om_E, \mc{F}_E)$, 
viz.~$\PP_p(\Omega_V \times \cdot)$.  In any case, since all our 
results are statements of almost-sure type, they hold $P_\om$-a.s.\ in the quenched setting.


We write $\pcsite$ (\resp, $\pcbond$) for the critical probability of site percolation
(\resp, bond percolation) on the lattice under consideration.

The structure of this paper is as follows. The main results for the \modelone\ and the \modeltwo\ are
presented in Section \ref{sec:main}, together with our conjectures for the full phase diagrams of the 
two models under consideration. 
The proofs of the main theorems are found in Sections \ref{sec:4}--\ref{sec:6}.

\section{The main results}\label{sec:main}

\subsection{The \modelone}
In this section all a.s.~statements are made with respect to the measure $\PP_p$.
\begin{theorem}\label{thm:compass}
	Let $d \ge 2$.
	For the \modelone\ of Example \ref{ex:compass}, there exist  $0<p_0(d) \le p_1(d)<1$ such that: 
	\begin{itemize}
		\item[(i)]  if $p\in (0,p_0(d))$, there exists a.s.\  a unique
		 infinite cluster of blue edges,
		\item[(ii)] if $p\in(p_1(d),1]$, there exists a.s.\  no infinite cluster of blue edges.
	\end{itemize}
\end{theorem}

See Sections \ref{sec:4} and \ref{sec:6} for the proofs.
We make the following further conjecture.

\begin{conjecture}
	\label{con:compass}
For the \modelone\ on $\ZZ^d$ with $d \ge 2$,
\begin{itemize} 
	\item[(i)]	there exists $\pc(d)\in (0,1)$ such that, for $p\in (0,\pc(d))$, there exists a.s.\ a
	unique infinite blue cluster, 
	while for $p>\pc(d)$ there exists a.s.\ no infinite blue cluster,
	\item[(ii)] $\pc(d)$ is strictly increasing in $d$,
	\item[(iii)] for given $d\ge 2$, the probability that the origin lies in an infinite blue cluster is non-increasing in $p$.
\end{itemize}
\end{conjecture}

Note in this example  with $d>2$ that, if we restrict the $d$-dimensional model to the two-dimensional
subgraph $\{0\}^{d-2}\times\ZZ^2$, the result is not 
the \modelone\ on $\ZZ^2$.   

\begin{remark}\label{rem:4}
Our numerical estimates of $\pc(2)$ and $\pc(3)$ in the \modelone\ are $\pc(2)\approx 0.505$ and $\pc(3)\approx 0.862$. 
\end{remark}

\subsection{The \modeltwo}
All a.s.~statements in this section are made with respect to the measure $\PP_{p,\lambda}$.

\begin{theorem}\label{thm:indep}
Let $d \ge 2$.
For the \modeltwo\ of Example \ref{ex:indep}, with parameters $(p,\lam)$,
	\begin{itemize}
		\item[(i)] if $\lambda < p/(2d-1)$, there is a.s.\ no infinite cluster of blue edges,
		\item[(ii)]  there exists an absolute constant $c>0$ such that,
		if $\la>c\log(1/q)$ where $p+q=1$, there exists a.s.\ a unique infinite blue cluster,
		\item[(iii)] there exists a Lipschitz continuous function $\olc:(0,1]\to(0,1]$, that
		satisfies $\olc(\pcsite)=1$, $\olc(1)=\pcbond$, and is strictly decreasing on $[\pcsite,1]$,
		  such that, for $p>\pcsite$ and $\lam>\olc(p)$,
there exists a.s.\ a unique infinite blue cluster.
	\end{itemize}
\end{theorem}

See Sections \ref{sec:5} and \ref{sec:6} for the proofs.
Figure \ref{fig:frontier} indicates what we believe to be the critical frontier for percolation in the \modeltwo\ in 
$2$ dimensions (based on simulations), and also what is proved.

\begin{figure}
	\centerline{
	\includegraphics[width=0.6\textwidth]{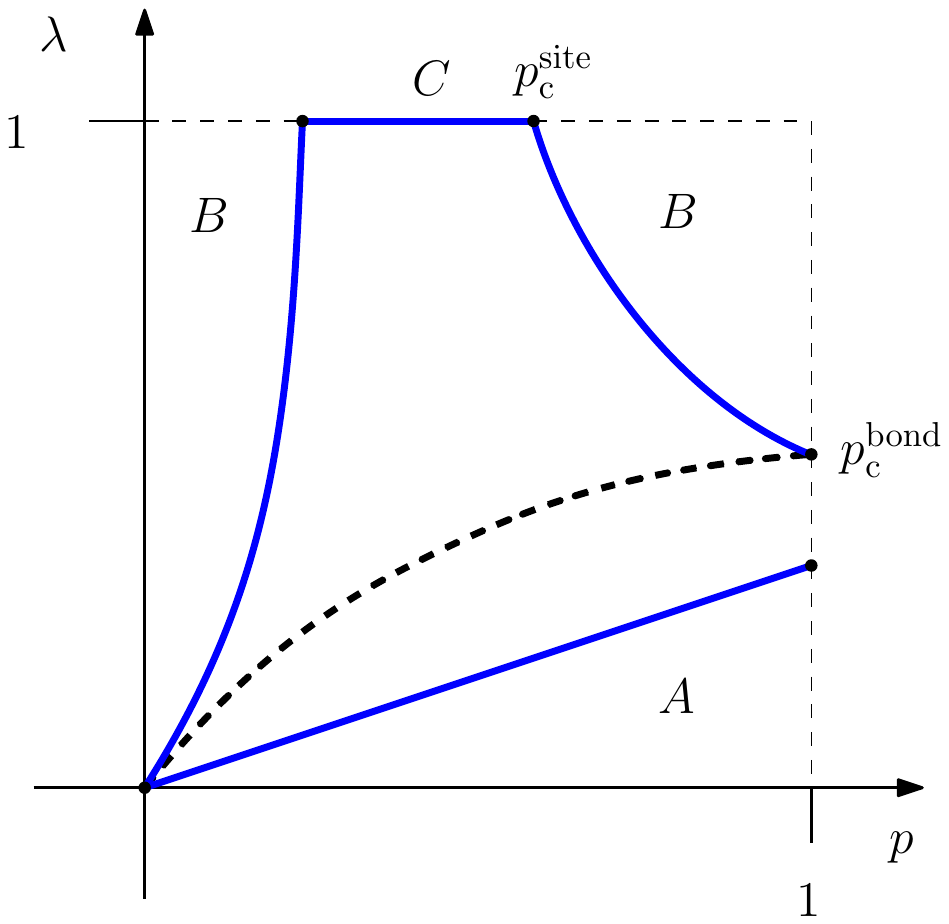}}
	\caption{A picture of the phase diagram for the \modeltwo\ in two dimensions.
	We conjecture there exists an infinite blue cluster for $(p,\lambda)$ above the dashed curve, and none below.  
	The solid lines indicate the regions identified in Theorem \ref{thm:indep}.  There is no percolation in region $A$
	where $\lambda\le p/3$, but there exists a unique infinite blue cluster in the two regions labelled $B$,
	and along the line $C$.}
	\label{fig:frontier}
\end{figure}

\begin{conjecture}
	\label{con:indep}
	For the \modeltwo\ on $\ZZ^d$ with $d \ge 2$, there exists $\lc=\lc(p,d)\in (0,1)$ such that:
	\begin{itemize}
		\item[(a)]  $\lc(\cdot,d)$ is continuous and strictly increasing on $(0,1]$,
		\item [(b)] for $p>0$ and $\lambda<\lc(p,d)$, there exists a.s.~no infinite blue cluster,
		 \item[(c)] for $p>0$ and $\lambda>\lc(p,d)$, there exists a.s.\ a unique infinite blue cluster, 
	\end{itemize}
Moreover, for each fixed $\lambda$, the probability that the origin is in an infinite blue cluster is non-increasing in $p$.
\end{conjecture}	

\begin{remark}[Added on revision]
In their recent preprint \cite{HU}, Hil\'ario and Ungaretti have proved the existence 
of such a critical function
$\lc$ with several properties including that $\lc(p)\le \pcH$ for $p\in(0,1]$ and $d \ge 2$,
where $\pcH$ is the critical probability of bond percolation on the hexagonal lattice.
\end{remark}

When $p=1$, the process is simply bond percolation with edge-density $\la$, 
whence $\lc(1,d)=\pcbond(d)$, the critical probability
of bond percolation on $\ZZ^d$. If $\lc$ is continuous on a neighbourhood of $p=1$, it then follows
that $\lc(p)\to \pcbond$ as $p\to 1$.

\begin{remark}
	  Our numerical estimate for the critical frontier of the \modeltwo\ is sketched 
	  as the dashed line in Figure \ref{fig:frontier}. 
\end{remark}

Consider the \modeltwo\ for given $(p, \la)$, and let $2\le d<d'$.
We may embed $\ZZ^d$ into $\ZZ^{d'}$ by identifying
$\ZZ^d$ with $\{0\}^{d'-d}\times \ZZ^d$. For given $(p, \la)$, the \modeltwo\
on $\ZZ^{d'}$, when restricted to this image of $\ZZ^d$, is simply a $d$-dimensional \modeltwo.
It follows in particular that $\lc(p,d)$ is non-increasing in $d$ (assuming $\lc$ exists).

\section{ Proof of Theorem \ref{thm:compass}(\textup{ii})}\label{sec:4}

We show that the \modelone\ and a \lq corrupted compass model' 
(see \cite[Sect.\ 3]{compass}) may be coupled in such a way that the first is a subset of the second.
It is known that the second does not percolate for sufficiently large values of $p$.

Let $\om\in\Om_V$ be a percolation configuration, with
occupied set $\eta=\eta(\om)$, and let a random feasible pair $f_v$ be as in Example \ref{ex:compass}.
The corrupted compass model is constructed as follows from $\om$ and the $f_v$. 
Each edge $e=(u,v)$ is declared to be \emph{turquoise} if
there exists $x\in\{u,v\}$ such that: either $x$ is unoccupied,
or $x$ is occupied and $e \in f_x$.

Let $T$ be the set of turquoise edges, and $B$ the set of blue edges in the \modelone.  We claim that
\begin{itemize}
	\item[(a)] $B\subseteq T$, and 
	\item[(b)] there exists $p_0=p_0(d)\in(0,1)$ such that if $p>p_0$, $T$ contains a.s.\ no infinite cluster.
\end{itemize}
Clearly, (a) and (b) imply the result. 

To verify (a), let $e\in B$, and suppose without loss of generality that $e=(o,e_1)$.   
Let 
\[
\ell_e=\inf\{k\ge 0: -ke_1\in \eta\}, \qq r_e=\inf\{k\ge  1:ke_1\in \eta\},
\]
and write  $x=-\ell_ee_1$ and $x'=r_e e_1$ for the closest occupied sites to the left and right of $o$, respectively
(we allow $x=o$).  Since $e\in B$, either $f_{x}=(x,x')$ or $f_{x'}=(x',x)$, or both.  
Without loss of generality, we assume $f_x=(x,x')$.  Then $(x,x+e_1)\in T$.  
Since the sites $\{x+ke_1: k=1,2,\dots,r_e+\ell_e-1\}$ are unoccupied, we have that $(x+ke_1,x+(k+1)e_1)\in T$ 
for $k \in \{1,\dots, r_e+\ell_e-1\}$.  In particular, $e\in T$.

To verify (b) we note that $T$ is precisely the set of edges of the corrupted compass model 
of \cite[Sect.\ 3]{compass} with corruption probability $p_{\{x\}}=1-p$ for each vertex $x\in \ZZ^d$ (that is,
each unoccupied vertex is corrupted).  As in \cite[Sect.\ 4.2]{compass},
$T$ does not percolate when the three eigenvalues $\kappa$ of the matrix 
\begin{equation}
\mathbf{M} = 
(2d-1)\begin{pmatrix}
1-p& 1-p&(1-p)/(2d)\\
1& 1/(2d) & 0\\
0 & 1 & 1/(2d)
\end{pmatrix},
\end{equation}
satisfy $|\kappa|<1$. (See  \cite[eqn (4)]{compass}, 
noting that the parameter $p$ therein is the so-called corruption probability, which is the current $1-p$.)  
As $p\to 1$, the eigenvalues of $\mathbf{M}$ approach their values with $p=1$,
namely $(2d-1)/(2d)$ and $0$, as required.
\qed

\section{Proof of Theorem \ref{thm:indep}(\textup{i, iii})}\label{sec:5}

We begin with an observation on the a.s.\ uniqueness of infinite clusters, when they exist.
There is a general theorem due basically to Burton and Keane \cite{BK,BK2} 
which states that, for any translation-invariant
probability measure $\mu$ on $\Xi=\{0,1\}^{\EE^d}$ 
with the finite energy property, if there exists an infinite cluster 
then it is a.s.\ unique. See also \cite[p.\ 42]{GHM}. A probability measure $\mu$ on $\Xi$ is
said to have the \emph{finite energy property} if
\begin{equation}\label{fin-en}
\mu(e\text{\rm\ is blue}\mid \sT_e)\in(0,1)\q \mu\text{-a.s.}, \qq e \in \EE^d,
\end{equation}
where $\sT_e$ is the $\sigma$-field generated by the colours of edges other than $e$.

The \modeltwo\ measure is evidently translation-invariant, and we state the finite energy property as a proposition.
Let $\mathbb{P}_{p,\lambda}$ be as in \eqref{joint} for the independent model.


\begin{proposition}[Finite energy property]\label{prop:BK}
Let $(p,\lam)\in (0,1]\times (0,1)$.
	 For  $e\in\EE^d$, we have that
\eqref{fin-en} holds with $\mu=\PPpl$.
\end{proposition}

In contrast, the one-choice measure 
does not have the finite energy property, and an
adaptation of the methodology of \cite{BK} will be needed for that case; 
see  Section \ref{sec:6}.

\begin{proof}
Write $e=(o,e_1)$, and let $O$ be the event that both $o$ and $e_1$ are occupied.
Let $E_0$ denote the set of edges incident to $o$ but not $e_1$, and 
 $E_1$  the set of edges incident to $e_1$ but not $o$.
Let $\mc{S}_i$ denote the random set of edges in $E_i$ that are blue.  Then 

\begin{equation}\label{eq:fin-en2}
\PPpl(O\mid \mc{T}_e)=\PPpl(O\mid \mc{S}_0,\mc{S}_1),
\end{equation}
since, given  $\mc{S}_o$ and $\sS_1$, the states of edges other than $E_o\cup E_1\cup\{e\}$
(and sites other than $o$, $e_1$) are conditionally independent of $O$.  

By conditional probability, for $S_i \subseteq E_i$,
\begin{align*}
\PPpl(O\mid \mc{S}_0=S_0,\,\mc{S}_1=S_1)&=\frac{\PPpl(\mc{S}_0=S_0,\,
\mc{S}_1=S_1\mid O)\PPpl(O)}{\PPpl(\mc{S}_0=S_0,\,\mc{S}_1=S_1)}\\
&\ge \PPpl(\mc{S}_0=S_0,\,\mc{S}_1=S_1\mid O)p^2.
\end{align*}
For $S_i\subseteq E_i$,
we have that
$$
\PPpl(\mc{S}_0=S_0,\,\mc{S}_1=S_1\mid O)>0.
$$
Since there are only finitely many choices for $S_0$, $S_1$, there exists $c_0(p,\lambda)>0$ such that 
\begin{align*}
\PPpl(O\mid \mc{S}_0=S_0,\,\mc{S}_1=S_1)\ge c_0(p,\lambda).
\end{align*}
By \eqref{eq:fin-en2}, $\PPpl(O\mid \mc{T}_e)\ge c_0(p,\lambda)$ a.s.

Now, a.s.,
\begin{align*}
\PPpl(e \text{ is blue}\mid \mc{T}_e) &\ge
\PPpl(e \text{ is blue},O\mid \mc{T}_e)\\
&=\PPpl(e \text{ is blue}\mid O,\mc{T}_e)\PPpl(O\mid \mc{T}_e)\\
&=\lambda \PPpl(O\mid \mc{T}_e)\ge \lambda c_0(p,\lambda),
\end{align*}
and, similarly, 
\begin{equation*}
\PPpl(e \text{ is not blue}\mid \mc{T}_e)\ge (1-\lambda) c_0(p,\lambda).
\end{equation*}
This proves \eqref{fin-en}.
\end{proof}

\begin{proof}[Proof of Theorem \ref{thm:indep}(i)]
Let $B_o$ be the set of vertices in the blue cluster at the origin $o$.
We shall bound $|B_o|$ above (stochastically)
by the total size of a certain branching process with two types of particle: 
type U with mean 
family-size $2\lam(d-1)/p$, and type O with mean family-size
$\lam(2d-1)/p$. This is an elaboration of the argument used in
\cite{H57} to prove that $\pc\ge 1/(2d-1)$ for bond percolation on $\ZZ^d$.

We outline the proof as follows. Suppose that $o$ is occupied 
 (the argument is similar if $o$ is unoccupied).
We explore the $2d$ directions incident to $o$, 
and count the number of sites on any blue segments touching $o$
(excluding $o$ itself).
The mean number of such sites is $2d\lam/p$, and this set of sites contains both occupied and unoccupied sites.
We now iterate the construction starting at these new sites. Each occupied site gives rise on average 
to no more than $\mu_1:=(2d-1)\lam/p$ new sites, and each unoccupied site no more than $\mu_2:=2(d-1)\lam/p$ new sites.
After each stage, we will have constructed some subset of $B_o$. 
As the construction develops, we will encounter further un/occupied sites each having a potential for further extensions.
During each stage, starting at some site $v$, the mean number of new sites is no greater than $\mu_1$ if
$v$ is occupied,
and $\mu_2$ if $v$ is unoccupied; the true conditional expectations (given the past history of the process) 
will typically be less than the $\mu_i$,
since some sites thus encountered will have been counted earlier.
  
In order to bound $|B_o|$ above, we consider the process with two types of particle: with type O 
	corresponding to occupied sites and type U corresponding to unoccupied sites. 
Assuming that this process may be taken to be a branching process,
it follows by standard theory that, if $\mu_1,\mu_2<1$, then the total size of the branching process
is  a.s.\ finite. This implies in turn that $\PP(|B_o|=\oo) = 0$ if $\mu_1,\mu_2<1$, which is to say if
$\lam<p/(2d-1)$. A more precise condition on the pair $\lam$, $p$ may be obtained by considering the
above branching process as a $2$-type process, and using the relevant extinction theorem 
(see, for example, \cite[Chap.\ V.3]{AN}).

The above argument  may be written out formally. There are some complications arising from the 
conjunction of probability
and combinatorics, but these do not interfere with the conclusion.
\end{proof}

\begin{proof}[Proof of Theorem \ref{thm:indep}(iii)]

Consider the following \lq mixed' percolation model on $\ZZ^d$ where $d \ge 2$.
Let $\lam, p\in[0,1]$. Each site is occupied independently with probability $p$. Each edge whose endvertices are both occupied is then occupied with probability $\lam$, independently of all other edges. If one or both of the ends of an edge is unoccupied then the edge itself cannot be occupied.
Let $C$ be the occupied cluster containing the origin,
and define the \emph{percolation probability} $\theta(\lam,p)=\PP(|C|=\oo)$.
Evidently, $\theta$ is non-decreasing in $\lam$ and $p$.
The \emph{critical curve} is defined by:
\begin{equation}\label{eq:crit}
\olc(p)=\sup\bigl\{\lam\in[0,1]: \theta(\lam,p)=0\bigr\},\qq p \in [0,1].
\end{equation}

Some basic properties of $\olc$ are proved in \cite[Thm 1.4]{CS}
as an application of the differential-inequality method of \cite{AG}
(see also \cite[Sect.\ 3.2]{G99}), namely the following:
\begin{letlist}
\item $\olc(p)=1$ for  $0\le p\le \pcsite$,
\item $\olc(1)=\pcbond$, 
\item $\olc$ is Lipschitz continuous and strictly decreasing on $[\pcsite,1]$.
\end{letlist}

We claim that the \modeltwo\ with parameters $(p,\lam)$ is stoch\-astically greater than
the mixed percolation model with parameters $(p,\lam)$.
To see this, consider the \modeltwo\ with parameters $(p,\lam)$, and let $G$ be the subset of blue edges with
the property that both their endvertices are occupied.  The law of $G$ is 
that of the mixed percolation model, and  the stochastic ordering
follows from the fact that $G\subseteq B$. 
It follows in particular that there exists a.s. an infinite blue cluster when $\lam>\olc(p)$.
\end{proof}

\section{Proofs of Theorem \ref{thm:compass}(\textup{i}) and Theorem \ref{thm:indep}\textup{(ii)}}
\label{sec:6}

We first prove the a.s.\ uniqueness of the infinite blue cluster in the
\modelone, when such exists.
The probability measure governing the \modelone\ may be constructed as follows.
As before, we write $\pp$ for product measure with density $p$ on $\Om_V$. 
For each $v \in \ZZ^d$, we choose a random coordinate direction having law the
product measure  $\Pi= \prod_{v\in \ZZ^d} \pi$, where $\pi$ is 
the uniform probability measure on 
the set $\sD=\{\pm e_i: i=1,2,\dots,d\}$. Let $\Ppp=\M{\pp} \times \Pi$ be the measure
on $\Om_V\times \sD$ that governs the \modelone.  

As explained at the beginning of Section \ref{sec:5},
proofs of uniqueness usually require that the relevant probability law
satisfy the finite-energy property \eqref{fin-en}.
The one-choice measure does not have the 
finite-energy property nor the \emph{positive} finite energy property of \cite{gandolfi}, as the following indicates (working under the annealed measure).   
\begin{example}
In $d=2$ dimensions consider the following local  configuration of edges: 
\begin{itemize}
\item[-] edges in the unit square with the origin $(0,0)$ as the bottom right corner are all blue,
\item[-] edges in the unit square with $(1,0)$ as the bottom left corner are all blue,
\item[-] all other edges incident to these squares (except the edge $e$ from $(0,0)$ to $(1,0)$, which we do not declare the state of) are not blue. 
\end{itemize}
It is an easy exercise to verify that the above configuration of blue and non-blue edges occurs with positive probability.  It also has the property that all vertices in the two squares must be occupied a.s., and that in each square the ``one choice'' made by each vertex gives an oriented cycle around the square.  In particular the occupied vertices $(0,0)$ and $(1,0)$ have a.s.~chosen one of the two edges in their relevant unit square that are incident to it.  It follows that a.s.~the edge $e$ is  not blue.
\end{example}
Notwithstanding, the arguments of \cite{BK} may be adapted (we do so in the proof of the following theorem) to obtain uniqueness
for the \modelone.

\begin{theorem}\label{thm:one-uniq}
Let $d \ge 1$, and let $N$ be the number of infinite blue clusters in the 
\modelone\ on $\ZZ^d$. Either $\Ppp(N=0)=1$
or $\Ppp(N=1)=1$.
\end{theorem}

\begin{proof}
We follow \cite{BK}, in the style of \cite[Sect.\ 8.2]{G99}. There are three 
statements to be proven, as follows:
\begin{Alist} 
\item $N$ is $\Ppp$-a.s.\ constant,
\item $\Ppp(N\in\{0,1,\oo\})=1$,
\item $\Ppp(N=\oo)=0$.
\end{Alist}

\smallskip
\noindent
\emph{Proof of \emph{A}.}
The probability measure $\Ppp$ is a product measure on $\Om_V \times \sD$,
and is therefore ergodic. The random variable $N$ is a translation-invariant function
on $\Om_V\times \sD$, and therefore $N$ is $\Ppp$-a.s.\ constant.

\smallskip
\noindent
\emph{Proof of \emph{B}.}
Suppose $\Ppp(N=k)=1$ for some $2\le k<\oo$, 
from which assumption we will obtain a contradiction.
As in \cite{G99}, we  work with
boxes of $\ZZ^d$ in the $L^1$ metric. For $m \ge 1$, let 
$$
D_m=\left\{x=(x_1,x_2,\dots, x_d)\in \ZZ^d: \sum_i |x_i| \le m\right\},
$$
considered 
as a   subgraph of the infinite lattice. The \emph{boundary}
of $D_m$ is the set $\pd D_m=D_m \setminus D_{m-1}$.

Let $N_m$ be the number of infinite blue clusters that intersect $D_m$.
We choose $m <\oo$ such that  
\begin{equation}\label{eq:inf1}
\Ppp(N=k,\, N_m\ge 2)>0.
\end{equation}
For $m<n$, let $F_{m,n}$ be the event that there exists no blue segment 
that intersects
both $D_m$ and $\ZZ^d\setminus D_n$.
Since all blue segments intersecting $D_m$ are a.s.\ finite,
for $\eps>0$ there exists $n>m$ such that $\Ppp(F_{m,n})>1-\eps$. This
we combine with \eqref{eq:inf1}, with sufficiently small $\eps$, to deduce the
existence of $n>m$ such that
\begin{equation}\label{eq:inf2}
\Ppp(N=k,\, N_m\ge 2,\, F_{m,n})>0.
\end{equation} 

On the event that $N_m \ge 2$,  we let $I_1$, $I_2$ be distinct infinite blue clusters that 
intersect $D_m$.  
For $j=1,2$, choose a point $z_j\in I_j\cap \pd D_m$ with the property that $z_j$ lies
in an infinite  blue path $J_j$ that traverses a sequence of 
sub-intervals of blue segments only one of which, denoted $f_j$, 
intersects  $D_m$ (the point of intersection being  $z_j$). We denote
by $b_j=( y_j,z_j)$ the edge of $f_j$ 
that is incident with $z_j$ and lies outside $D_m$ and such that $y_j$ lies in an infinite blue path that is a subset of $J_j \setminus D_m$; 
if there exists a choice for $b_j$,  we choose one according to some fixed but arbitrary rule.
See Figure \ref{fig:uniq}.

\begin{figure}
	\includegraphics[width=0.6\textwidth]{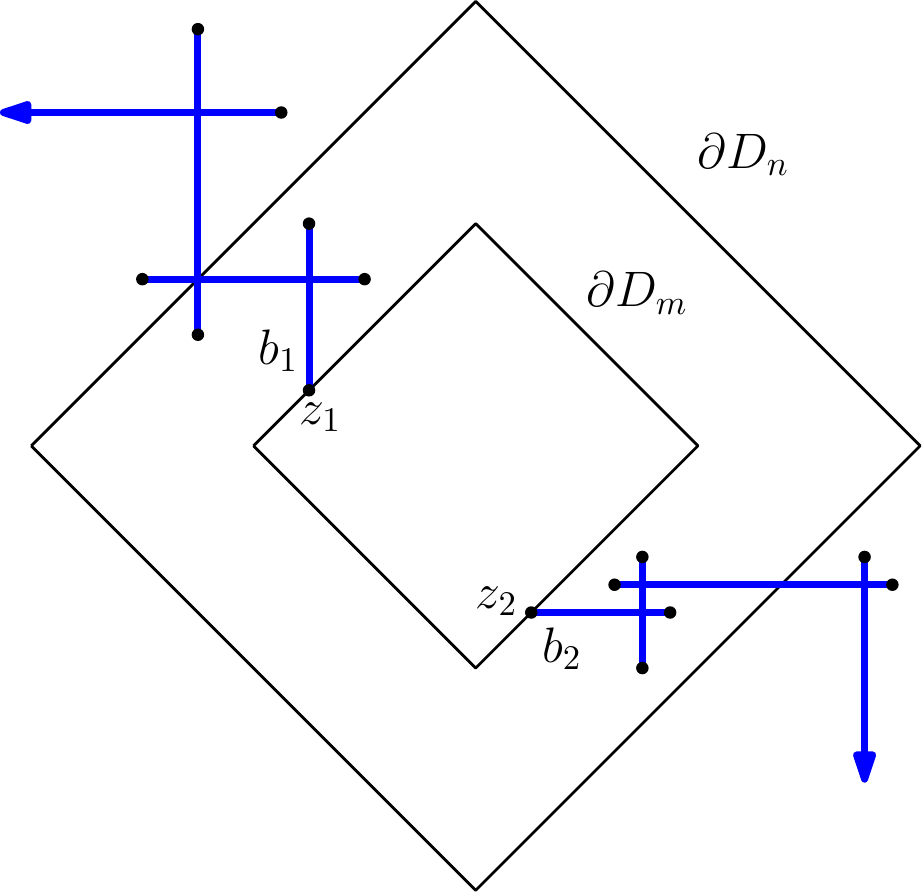}
	\caption{An illustration of the first part of the  proof of statement B in 
	Theorem \ref{thm:one-uniq}.
	The arrows indicate connections to infinity.
	}
	\label{fig:uniq}
\end{figure}

We now perform surgery within the larger box $D_n$, by altering the configuration of occupied sites and of chosen directions. This is done in three stages.

First, we designate each $z_j$ occupied (if not already occupied),
and we set its chosen direction as pointing outwards from $D_m$ along the
edge $b_j$, and we write $f_j'$ for the blue segment  containing $b_j$ in
the ensuing graph.  
If $z_j$ was not previously occupied, this may have the consequence of 
removing part of $f_j$ from the blue graph. For any blue segment 
$f\ne f_1',f_2'$, with
endvertices $u$, $v$, that
contains a point in $D_m$, we do the following.
\begin{letlist}
\item   If $u\in D_m$ (\resp, $v\in D_m$) and $u\ne z_1,z_2$
(\resp, $v\ne z_1,z_2$) we designate it  unoccupied.  
\item  If $u\notin D_m$ (\resp, $v \notin D_m$) is the endvertex of no blue segment 
other than $f$, we designate it unoccupied.
\item If $u\notin D_m$ (\resp, $v \notin D_m$) is the endvertex of some blue segment 
$f'\ne f$ such that $f'\subseteq \ZZ^d\setminus D_m$, we keep 
$u$ (\resp, $v$)  occupied and have it choose its direction along $f'$. 
\item If $u\notin D_m$ (\resp, $v \notin D_m$) is the endvertex of no blue segment 
$f'\ne f$ such that $f'\subseteq \ZZ^d\setminus D_m$, but is the endpoint of some 
$f''\ne f$ such that $f''\cap D_m \ne \es$, we designate it unoccupied.
\end{letlist}
At the completion of this stage, all vertices of $D_m$ except 
$z_1$ and $z_2$ are unoccupied, and moreover $z_1$ and $z_2$ are 
the only vertices of $D_m$ belonging to a blue segment.
We have not added any blue edges, but some may have been removed.  Thus
the clusters of $z_1$ and $z_2$ are still (infinite and) disjoint.  If with
positive probability this procedure gives a new $N\ne k$ then we have
obtained a contradiction since we have shown that $\mathbb{P}_{p,\lambda}(N\ne
 k)>0$ (we have only changed the states of sites and their chosen
directions within $D_n$).  Otherwise almost surely $N$ is still equal to $k$
after
this procedure.

\begin{figure}

\begin{tikzpicture}
	\tikzset{vert/.style={circle, fill=black, draw=black, inner sep=1pt, line width=0.5pt}}
	\draw (0,0) -- (4,4) -- (8,0) -- (4,-4) -- cycle;
	\node at (1,-1) [vert, label={[label distance=-5pt] below left:{$z_1$}}] {};
	\node at (6,2) [vert, label={[label distance=-5pt] above right:{$z_2$}}] {};
	\draw [line width=1.5pt, blue] (3,0) node [vert] {} -- (5,0) node [vert] {};
	\draw [line width=1.5pt, blue, decoration={markings, mark=at position 0.7 with {\arrow{latex}}}, postaction=decorate] (3.5,0.5) node [vert] {} -- (3.5,-1) node [vert] {};
	
	\draw [line width=1.5pt, blue, decoration={markings, mark=at position 0.5 with {\arrow{latex}}}, postaction=decorate] (3.5,-1) node [vert] {} -- (1,-1) node [vert] {};
	\draw [line width=1.5pt, blue, decoration={markings, mark=at position 0.7 with {\arrow{latex}}}, postaction=decorate] (4.5,-0.5) node [vert] {} -- (4.5,2) node [vert] {};
	\draw [line width=1.5pt, blue, decoration={markings, mark=at position 0.5 with {\arrow{latex}}}, postaction=decorate] (4.5,2) node [vert] {}-- (6,2) node [vert] {};
	\end{tikzpicture}
\caption{An illustration of the second part of the proof of B in Theorem \ref{thm:one-uniq}.}
\label{fig:uniq21}

\end{figure}
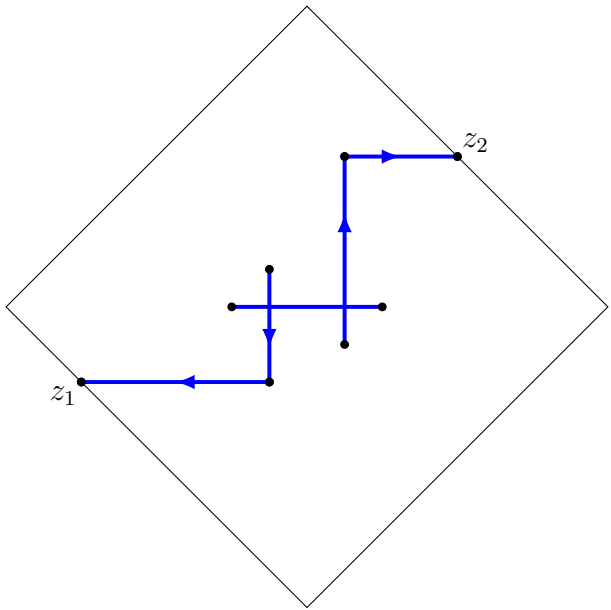

We now further alter the configuration within $D_m$ in order to
create a blue connection between $z_1$ and $z_2$. 
The following explanation is  illustrated in Figure \ref{fig:uniq21}.

Recall that the origin $o$ is the central vertex of $D_m$ and let $c_\pm = o\pm2e_1$. Set $c_+$ and $c_-$ to be occupied with their chosen directions pointing towards each other, so that the segment $S$ of four edges between them is blue. We now construct two disjoint paths which meet this segment and end at $z_1$ and $z_2$. This occurs in two steps.
\begin{itemize}
	\item[(1)] Choose one of the three points $\{o-e_1,o,o+e_1\}$ to connect to $z_1$, and another to connect to $z_2$. Call them $s_1$ and $s_2$. (For $m$ sufficiently large the choice does not matter, but a judicious choice may lead to simpler paths.)
	\item[(2.i)] Path $\rho_1$ is comprised of occupied vertices $\rho_{1,0}, \rho_{1,1},\dots,\rho_{1,k} (= z_1)$ and the unoccupied vertices between them. It starts at one of the $2d-2$ vertices $s_1\pm e_i$, $i\in\{2,\dots,d\}$, meeting $S$ only at the vertex $s_1$ (between $\rho_{1,0}$ and $\rho_{1,1}$). For $j=0,\dots,k-1$, each occupied $\rho_{1,j}$ has chosen direction pointing towards $\rho_{1,j+1}$.
	\item[(2.ii)] Path $\rho_2$ is comprised of occupied vertices $\rho_{2,0}, \rho_{2,1},\dots,\rho_{2,l} (= z_2)$ and the unoccupied vertices between them. It starts at one of the $2d-2$ vertices $s_2\pm e_i$, $i\in\{2,\dots,d\}$, meeting $S$ only at the vertex $s_2$ (between $\rho_{2,0}$ and $\rho_{2,1}$). For $j=0,\dots,l-1$, each occupied $\rho_{2,j}$ has chosen direction pointing towards $\rho_{2,j+1}$.
\end{itemize}
It may be checked that this construction is always possible for $m\geq5$, and that we have a unique path of blue edges from $z_1$ to $z_2$.

In summary, by altering the configuration within $D_n$,
the number of infinite blue clusters can be reduced by one. 
Since $D_n$ contains boundedly many
vertices, we deduce from \eqref{eq:inf2} that
$\Ppp(N=k-1)>0$, in contradiction of the assumption that $\Ppp(N=k)=1$.
Statement B follows.

\smallskip
\noindent
\emph{Proof of \emph{C.}}
Following \cite{BK}, the idea is to assume that $\Ppp(N\ge 3)=1$, to find \emph{three}
sites $z_1,z_2,z_3\in \pd D_m$
connected to disjoint infinite blue clusters of $\ZZ^d\setminus D_m$,
and then to perform surgery, as in the proof of B above, 
to show that $D_m$ contains, with strictly positive probability,
a vertex $t$ whose deletion breaks an infinite blue cluster into three 
disjoint infinite blue clusters. Such a vertex $t$ is called a \emph{trifurcation},
and a neat argument due to Burton and Keane \cite{BK},
based on the translation-invariance and polynomial growth of $\ZZ^d$, shows the
impossibility of their existence with strictly positive probability.
This provides the contradiction from which statement C follows.
It suffices, therefore, to show that, if $\Ppp(N\ge 3)=1$, then
there is a strictly positive probability of the existence of a trifurcation.   
It is not unusual in the related literature to close the proof with the above sketch, 
but we continue with some details. We hope that the outline explanation presented here 
will convince readers without burdening them  
with too many details. 

Fix $m \ge 1$; we shall require later that $m\ge M$ for some absolute constant $M$.
As in \eqref{eq:inf2}, we find $n>m$ such that
$$
\Ppp(N_m\ge 3,\ F_{m,n})>0,
$$
and, on the last event, we find distinct sites $z_1,z_2,z_3\in \pd D_m$
connected to disjoint infinite blue clusters of $\ZZ^d\setminus D_m$.
By altering the configuration on $D_n$ as in the proof of statement B,
we arrive at an event, with strictly positive probability, on which: (i)
the $z_i$ lie in infinite blue paths of 
$\ZZ^d\setminus D_m$, (ii) apart from the $z_i$, every site in $D_m$ is unoccupied, and (iii)
$D_m$ contains no blue edge.

The rest of this proof is devoted to describing how to alter the configuration inside $D_m$ 
in order to create a trifurcation. It is a simple generalisation of the proof of B above. We construct the same segment $S$ as before. Then instead of constructing two disjoint paths $\rho_1$ and $\rho_2$, we construct three disjoint paths $\rho_1,\rho_2$ and $\rho_3$. We begin in a similar way.
\begin{itemize}
	\item[(1)] Choose one of the three points $\{o-e_1,o,o+e_1\}$ to connect to $z_1$, another to connect to $z_2$, and the third to connect to $z_3$. Call them $s_1, s_2$ and $s_3$. (Again, for $m$ sufficiently large the choice does not matter, but a judicious choice may lead to simpler paths.)
	\item[(2.i--ii)] See the proof of B.
	\item[(2.iii)] Path $\rho_3$ is comprised of occupied vertices $\rho_{3,0}, \rho_{3,1},\dots,\rho_{3,h} (= z_3)$ and the unoccupied vertices between them. It starts at one of the $2d-2$ vertices $s_3\pm e_i$, $i\in\{2,\dots,d\}$, meeting $S$ only at $s_3$ (between $\rho_{3,0}$ and $\rho_{3,1}$). For $j=0,\dots,h-1$, each occupied $\rho_{3,j}$ has chosen direction pointing towards $\rho_{3,j+1}$.
\end{itemize}
It may be checked that this construction is always possible for $m\geq7$. The blue edges contain unique paths between $z_1,z_2$ and $z_3$, and the origin $o$ is a trifurcation.

\begin{figure}

\begin{tikzpicture}
	\tikzset{vert/.style={circle, fill=black, draw=black, inner sep=1pt, line width=0.5pt}}
	\tikzset{tri/.style={circle, fill=red, draw=red, inner sep=1.5pt, line width=0.5pt}}
	\draw (0,0) -- (4,4) -- (8,0) -- (4,-4) -- cycle;
	\node at (1,-1) [vert, label={[label distance=-5pt] below left:{$z_1$}}] {};
	\node at (6,2) [vert, label={[label distance=-5pt] above right:{$z_2$}}] {};
	\node at (7,1) [vert, label={[label distance=-5pt] above right:{$z_3$}}] {};
	\draw [line width=1.5pt, blue] (3,0) node [vert] {} -- (5,0) node [vert] {};
	
	\draw [line width=1.5pt, blue, decoration={markings, mark=at position 0.7 with {\arrow{latex}}}, postaction=decorate] (3.5,0.5) node [vert] {} -- (3.5,-1) node [vert] {};
	\draw [line width=1.5pt, blue, decoration={markings, mark=at position 0.5 with {\arrow{latex}}}, postaction=decorate] (3.5,-1) node [vert] {} -- (1,-1) node [vert] {};
	
	\draw [line width=1.5pt, blue, decoration={markings, mark=at position 0.7 with {\arrow{latex}}}, postaction=decorate] (4.5,-0.5) node [vert] {} -- (4.5,2) node [vert] {};
	\draw [line width=1.5pt, blue, decoration={markings, mark=at position 0.5 with {\arrow{latex}}}, postaction=decorate] (4.5,2) node [vert] {}-- (6,2) node [vert] {};
	
	\draw [line width=1.5pt, blue, decoration={markings, mark=at position 0.7 with {\arrow{latex}}}, postaction=decorate] (4,0.5) node [vert] {} -- (4,-1);
	\draw [line width=1.5pt, blue, decoration={markings, mark=at position 0.5 with {\arrow{latex}}}, postaction=decorate] (4,-1) node [vert] {} -- (6,-1);
	\draw [line width=1.5pt, blue, decoration={markings, mark=at position 0.5 with {\arrow{latex}}}, postaction=decorate] (6,-1) node [vert] {} -- (6,1);
	\draw [line width=1.5pt, blue, decoration={markings, mark=at position 0.5 with {\arrow{latex}}}, postaction=decorate] (6,1) node [vert] {} -- (7,1) node [vert] {};
	
	\node at (4,0) [tri] {};
\end{tikzpicture}
	\caption{An illustration of the proof of C in Theorem \ref{thm:one-uniq}. The central vertex (indicated in red) is a trifurcation.}
	\label{fig:uniq6}
\end{figure}
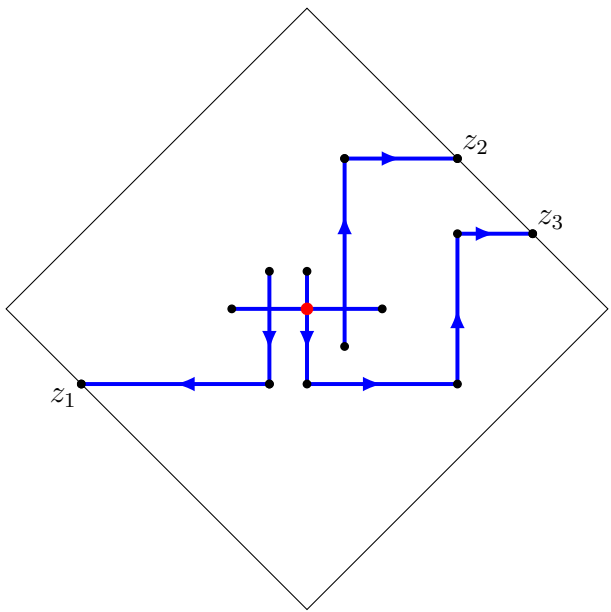

This construction is illustrated in Figure \ref{fig:uniq6}, and the proof is complete.
\end{proof}

We turn now to the proof of the existence of an infinite
blue cluster, subject to appropriate conditions.
Let $\om\in\Om_V$, and let $F(\om)$ be the set of feasible pairs.
The configuration space is 
$\Phi(\om)=\{0,1\}^{F(\om)}$, and for $\phi\in\Phi(\om)$, we call $f\in F(\om)$ (and the corresponding segment of $\ZZ^d$) 
blue if $\phi_f=1$.
We shall consider  probability measures $\mu$ on $\Phi(\om)$ that satisfy a somewhat less restrictive  condition
than those of the \modelone\ and the \modeltwo.

Let $\mu$ be a probability measure on $\Phi(\om)$ satisfying the following conditions:
\begin{Alist}
\item [C1.]
for any family $\{f_i=(u_i,v_i): i\in I\}$
of site-disjoint, feasible 
pairs, the events  
$\{f_i\text{ is blue}\}$, $i\in I$, are independent, and
\item[C2.]
there exists $\la=\la(\mu)\in(0,1)$ such that, for all $f \in F(\om)$, we have
\begin{equation}\label{eq:1}
\mu(f\text{ is blue})=\la.
\end{equation}
\end{Alist} 
Note that both the \modelone\ and the  \modeltwo\  satisfy C1 and C2 above.

Let $B$ be the set of  blue edges, so that $B$ gives rise to a subgraph
of $\ZZ^d$ (also denoted $B$) with vertex-set $\ZZ^d$ and edge-set $B$. 
We write $\PPpm$ for the law of $B$, and let $\la$ be
defined by \eqref{eq:1}.


Let $\psi(p,\mu)$ be the probability that $B$ contains an infinite cluster.
Here is the main result of this section.

\begin{theorem}\label{thm:1}
Let $d=2$ and suppose $\mu$ satisfies C1 and C2 above. There exists an absolute constant $c>0$ such that,
if $p=1-q\in(0,1)$ and $\lam>c\log(1/q)$, then $\psi(p,\mu)=1$.
\end{theorem}

Let $d \ge 2$ and $\ZZ^2_d:=\{0\}^{d-2}\times \ZZ^2$, considered as a subgraph of $\ZZ^d$.
Theorem \ref{thm:1} has the following consequences.

\begin{Alist}
\item[(1)] \emph{The \modelone\ on $\ZZ^d$ with $d\ge 2$.} 
The \modelone\ on $\ZZ^d$, when restricted to $\ZZ^2_d$, is no longer a  
\modelone\ (for example, there exist a.s.\ occupied sites with no incident blue edge).  
Nevertheless, the restricted process satisfies C1 and C2, above, with parameter $\la=\la(d)>0$ as given in \eqref{lambda_def_choice}.
By Theorem \ref{thm:1},
if $\la(d)>c\log(1/q)$, then $\ZZ^2_d$ contains an infinite blue cluster a.s., and  therefore so does $\ZZ^d$.
The condition on $\la(d)$ amounts to assuming $p<p_0(d)$ for some $p_0(d)>0$. 
Theorem \ref{thm:compass}(i) is proved.

\item[(2)] \emph{The \modeltwo.}
In this case, $\la>0$ is a parameter of the model. For $d\ge 2$, the model restricted to $\ZZ^2_d$ is the
two-dimensional \modeltwo\ with the same value of $\la$.  By Theorem \ref{thm:1},
$\ZZ^2_d$, and hence $\ZZ^d$ also, contains an infinite blue cluster a.s.\
whenever $\la>c\log(1/q)$. This proves Theorem \ref{thm:indep}(ii).

\end{Alist}

\begin{proof}[Proof of Theorem \ref{thm:1}.]
The proof is by comparison with a  supercritical, $1$-dependent site
percolation process on $\ZZ^2$, and proceeds via a block argument
not dissimilar to that used in \cite{BGN,GM}. The idea
is as follows. Let $p\in(0,1]$. We partition $\ZZ^2$ into blocks of
given side-length $6r$ where $r=r(p)$ satisfies $r(p)\to\oo$ as $p\downarrow 0$, and 
certain blocks will be called \lq good' (the meaning of this will be explained).
Then we show that the set of good blocks dominates (stochastically)
a $1$-dependent site percolation process with a density that
can be made close to $1$
by making $p$ sufficiently small (and $r$ correspondingly large). Finally, we show that,
if there is an infinite cluster of good blocks, then there is necessarily an infinite
blue cluster in the original lattice $\ZZ^2$.

Write $e_1=(1,0)$, $e_2=(0,1)$, 
so that the 
four neighbours of $x \in \ZZ^2$ are $x\pm e_i$, $i=1,2$.
Let $p=1-q\in(0,1]$.
We begin by defining the relevant block events. Let $r \ge 1$, to
be chosen later, and let 
$\Lazz=[-3r,3r]^2 \cap \ZZ^2$. For $x=(x_1,x_2)\in\ZZ^2$, let
$\La_x=6rx+\La_o$, the translate of $\Lazz$ by $6rx$. 
The boxes $\{\La_x: x \in \ZZ^2\}$ form a square paving of $\ZZ^2$.

\begin{figure}
	\includegraphics[width=0.4\textwidth]{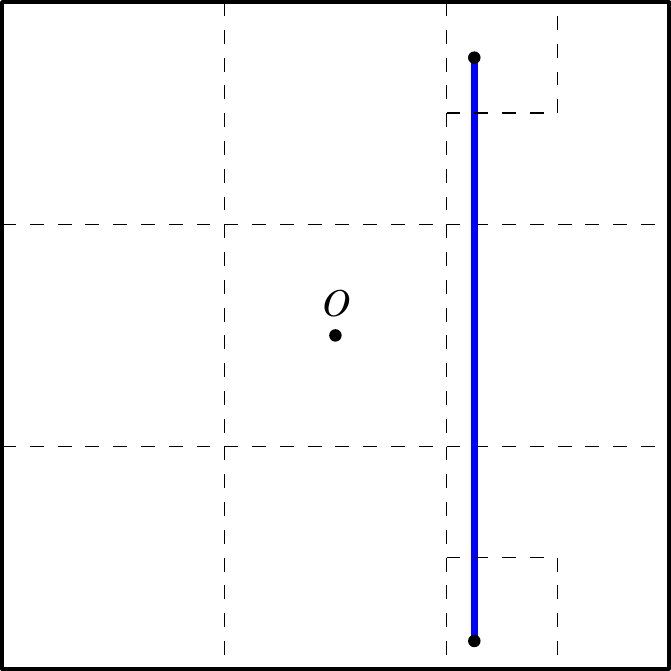}\q\includegraphics[width=0.4\textwidth]{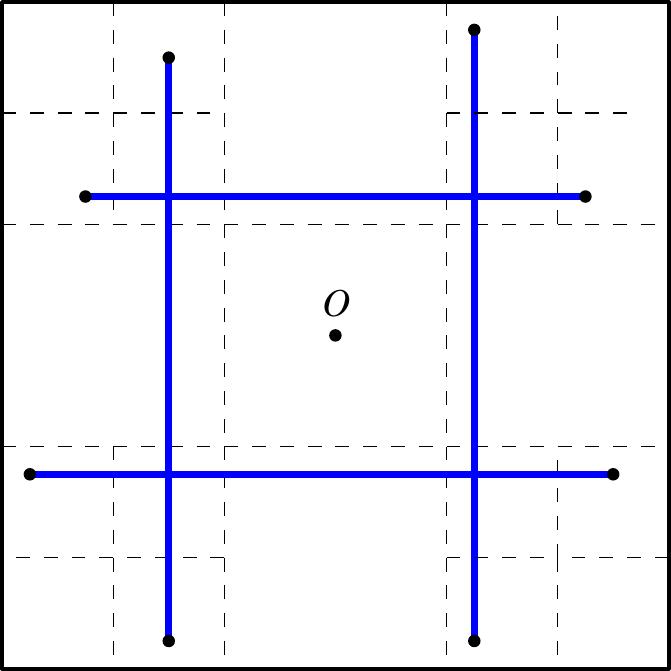}
	\caption{The events $A_{e_1}$ and $A(o)$.}
	\label{fig:event1}
\end{figure}

Consider first the special case $x=o$. Let
$$
S_1=[r,2r)\times[2r,3r),\q
S_2=[r,2r)\times(-3r,-2r],
$$
as illustrated in Figure \ref{fig:event1},  and let
$A_{e_1}$ be the event that there exist occupied vertices $s_1\in S_1$,
$s_2\in S_2$, such that the unordered pair $(s_1,s_2)$ is feasible and blue.
The probability of $A_{e_1}$ may be calculated as follows.
Let $L$ be the line-segment $[r,2r)\times\{0\}$. For $(k,0)\in L$,
let $D_k$ be the event that 
\begin{letlist}
\item  when proceeding north from $(k,0)$, the first occupied vertex $v$
encountered (including possibly $(k,0)$ itself) is at distance between $2r$ and $3r-1$ from $L$,  and
\item when proceeding south from $(k,0)$, the first occupied vertex $w$
encountered (excluding $(k,0)$) is at distance between $2r$ and $3r-1$ from $L$, and
\item the feasible pair $(v,w)$ is blue.
\end{letlist}
Now, $A_{e_1}$ is the disjoint union of the $D_k$ for $(k,0)\in L$.
Therefore,
\begin{align}\label{eq:3}
1-\PPpm(A_{e_1}) &=\bigl[1-\la (q^{2r}-q^{3r})(q^{2r-1}-q^{3r-1})\bigr]^r\\
&=\left[1-\frac \la q [q^{2r}(1-q^r)]^2\right]^r.\nonumber
\end{align}
We choose $r=r(p)$ to satisfy
\begin{equation}\label{eq:4}
q^{r}=\tfrac12,\qq \text{that is}\q r=\frac1{\log_2(1/q)}.
\end{equation}
Note that when $p$ is small, $q$ is close to 1,
and so $r$ is large.  A small correction is necessary in order that $r$ be an integer, but we shall
overlook this for ease of notation. By \eqref{eq:3}, there exists an absolute constant
$c_1>0$ such
that
\begin{equation}\label{eq:5}
1-\PPpm(A_{e_1}) \le e^{-c_1\la r}.
\end{equation}

Let $A_{-e_2}, A_{-e_1}, A_{e_2}$ be the 
respective events corresponding to $A_{e_1}$ after
$\Lazz$ has been rotated clockwise by multiples of $\pi/2$. 
By \eqref{eq:5} and symmetry,
\begin{equation}\label{eq:6}
1-\PPpm(A(o)) \le 4e^{-c_1\la r},
\end{equation}
where the event $A(o):=A_{e_1} \cap A_{-e_2}\cap A_{-e_1}\cap A_{e_2}$ 
is illustrated 
in Figure \ref{fig:event1}. On $A(o)$, there exists a connected blue 
subgraph with large diameter in the box $\Lazz$.

\begin{figure}
	\includegraphics[width=0.8\textwidth]{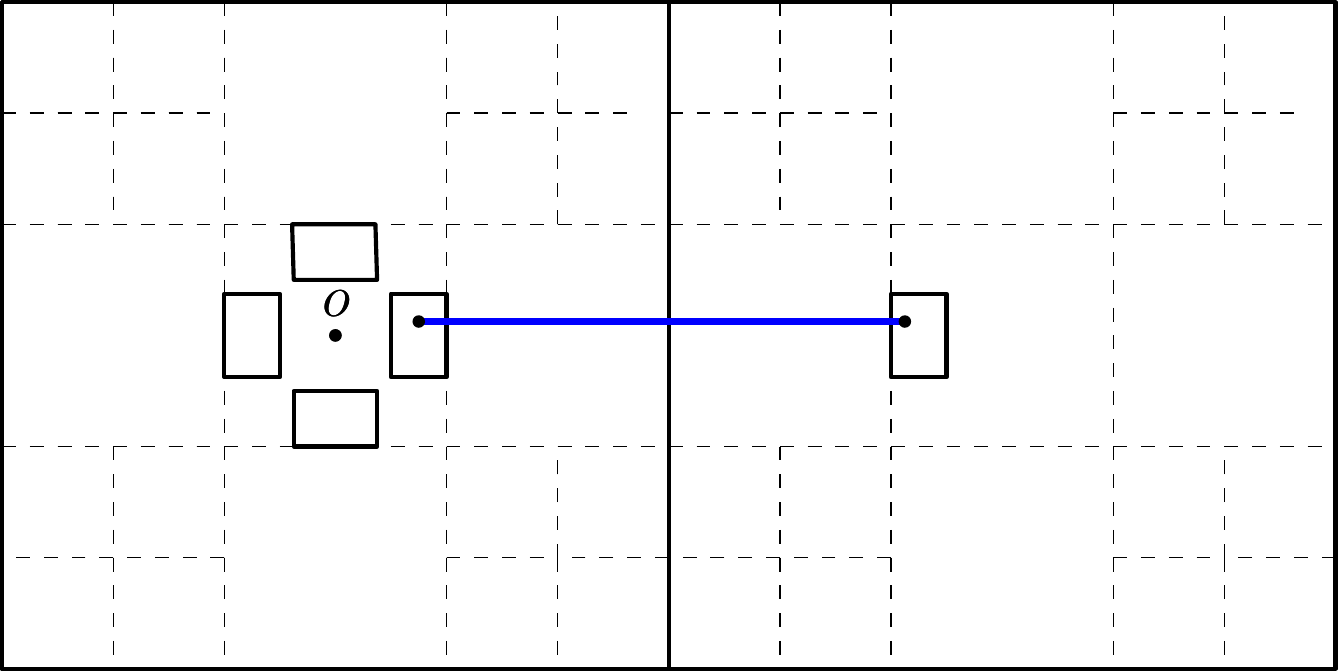}
	\caption{The regions $T_{e_i}$, $6re_1+T_{-e_1}$, and the event $C_{e_1}$. }
	\label{fig:Tregions}
\end{figure}

We turn next
to blue connections between neighbouring boxes, beginning
with connections between $\Lazz$ and $\Laoz$.
Let 
\begin{align*}
T_{e_1}=(\tfrac13 r,r]\times [-\tfrac13 r,\tfrac13r],\q
&T_{-e_2}=[-\tfrac13 r,\tfrac13 r]\times [-r, -\tfrac13 r),\\
T_{-e_1}=[-r,-\tfrac13 r) \times [-\tfrac13 r,\tfrac13r],\q
&T_{e_2}= [-\tfrac13 r,\tfrac13 r]\times (\tfrac13 r,r] ,
\end{align*}
as in Figure \ref{fig:Tregions},  noting that the $T_{\pm e_i}$ are disjoint; 
we assume for simplicity of notation that
$r$ is a multiple of $3$. Let $C_{e_1}$ be the event
that there exists a blue, feasible pair $f=(u,v)$ with $u\in T_{e_1}$, 
$v\in 6r e_1+T_{-e_1}$.  Similar to \eqref{eq:3}--\eqref{eq:5} we have that
\begin{align}\label{eq:7}
1-\PPpm(C_{e_1}) &= \left[1-\frac \la q [q^{2r}(1-q^{2r/3})]^2\right]^{2r/3}\\
&\le e^{-c_2\la r},\nonumber
\end{align}
for some absolute constant $c_2>0$.
Let $C_{\pm e_i}$ be defined as $C_{e_1}$ but with respect to blue connections 
between $T_{\pm e_i}$ and $\pm 6r e_i+T_{\mp e_i}$. By symmetry, \eqref{eq:7}
holds with $C_{e_1}$ replaced by $C_{\pm e_i}$, and by \eqref{eq:7} the event
$C(o):=C_{e_1} \cap C_{-e_2}\cap C_{-e_1}\cap C_{e_2}$ (depicted in Figure \ref{fig:good_event}) satisfies
\begin{equation}\label{eq:9}
1-\PPpm(C(o)) \le 4e^{-c_2\la r},
\end{equation}

We are now ready to define the block event associated with the vertex $o$.
We call $o$ \emph{good} if the event $A(o) \cap C(o)$ occurs (see Figure \ref{fig:good_event}).
By \eqref{eq:5} and \eqref{eq:9},
\begin{equation}\label{eq:8}
\PPpm(\text{$o$ is good}) \ge 1 - 4e^{-c_1\la r} - 4e^{-c_2\la r}.  
\end{equation}

\begin{figure}
	\includegraphics[width=0.6\textwidth]{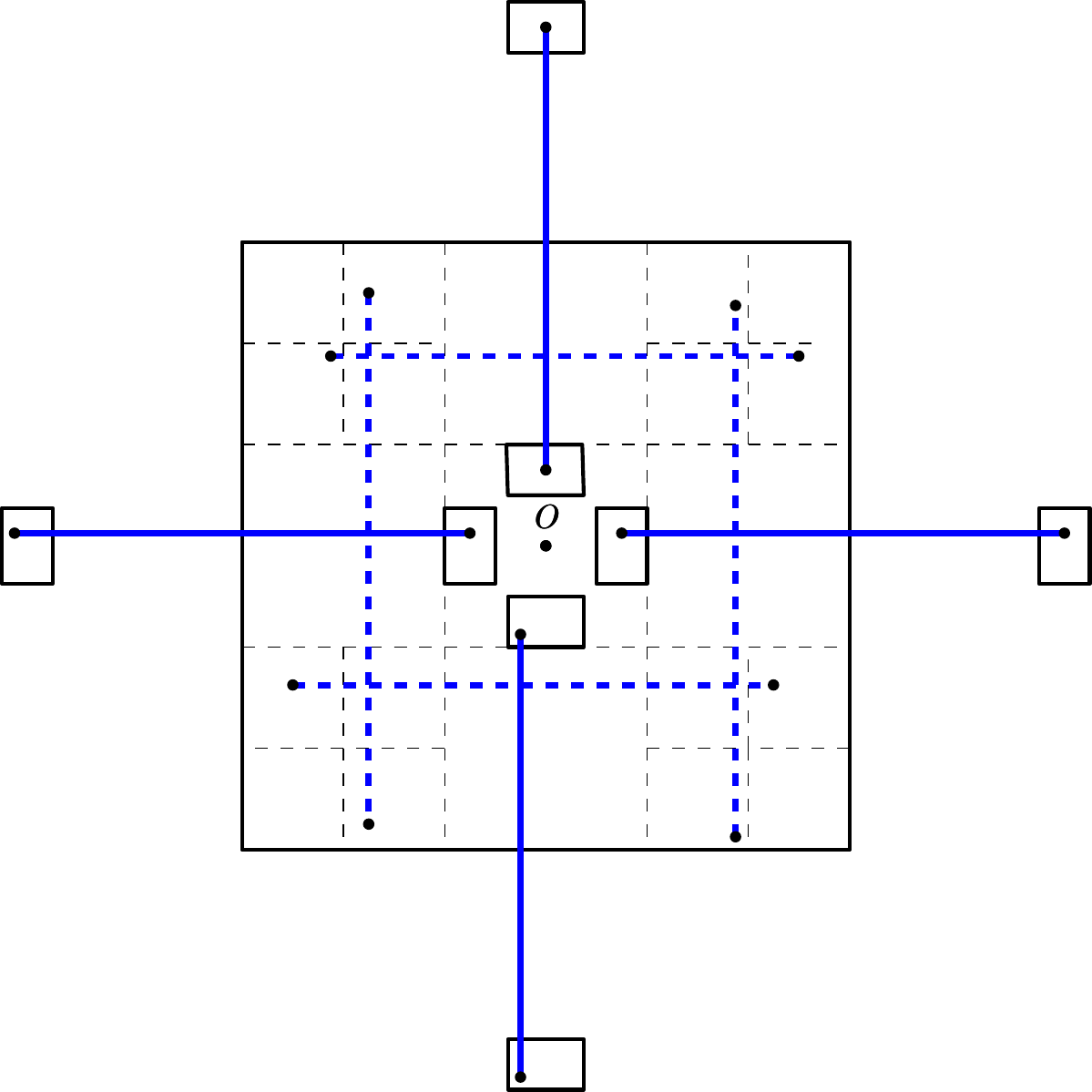}
	\caption{The event $C(o)$ is depicted in solid lines, and $A(o)$ in dashed lines. }
	\label{fig:good_event}
\end{figure}

Let  $x \in \ZZ^2$, and let $\tau_x$ be the translation on $\ZZ^2$ by $x$, so that
$\tau_x(y)=x+y$.
This induces a translation on $\Om_V$, also denoted $\tau_x$, by: 
for $\om=\{\om_v: v \in \ZZ^2\}\in\Om_V$, 
we have $\tau_x(\om)= \{\om_{v-x}: v \in \ZZ^2\}$.
The vertex $x \in \ZZ^2$ is declared \emph{good} if $o$ is good in the configuration
$\tau_{6rx}(\om)$, and we write $G_x$ for the event that $x$ is good.
By \eqref{eq:8},
\begin{equation}\label{eq:10}
\PPpm(G_x) \ge 1 - 4e^{-c_1\la r} - 4e^{-c_2\la r} , 
\qq x \in \ZZ^2.
\end{equation}

Let $d(u,v)$ denote the graph-theoretic 
distance from vertex $u$ to vertex $v$.
By examination of the definition of the events $G_x$, we see that their law 
is $1$-dependent, in that, for $U,V \subseteq \ZZ^2$, the families
$\{G_u: u \in U\}$, and $\{G_v: v\in V\}$ are independent
whenever $d(u,v) \ge 2$ for all $u\in U$, $v\in V$. 
By \cite[Thm 0.0]{LSS} (see also \cite[Thm 7.65]{G99}),
there exists $\pi<1$ such that: there exists almost surely an infinite cluster of
good vertices of $\ZZ^2$ whenever $\PPpm(G_0)\ge \pi$.

We choose $c>0$ such that
$1 - 4e^{-c_1\la r} - 4e^{-c_2\la r} > \pi$ whenever $\la r>c$.
By \eqref{eq:4}, there exists a.s.\ an infinite good cluster in the block lattice $6rx\ZZ^2$ if
$\la>c\log_2(1/q)$, which is to say that $q^c2^\la>1$.

By considering the geometry of the block events, we see that any cluster of
good vertices of $\ZZ^2$ gives rise to a cluster of blocks whose union contains 
a blue cluster intersecting every such block. The claim of
the theorem follows.
\end{proof}

\section*{Acknowledgements} 
NRB is supported by grant DE170100186 from the Australian Research Council.
MH is supported by Future Fellowship FT160100166 from the Australian Research Council.  The authors are grateful to a colleague and two referees for their comments.

\bibliographystyle{amsplain}


\providecommand{\bysame}{\leavevmode\hbox to3em{\hrulefill}\thinspace}
\providecommand{\MR}{\relax\ifhmode\unskip\space\fi MR }
\providecommand{\MRhref}[2]{%
  \href{http://www.ams.org/mathscinet-getitem?mr=#1}{#2}
}
\providecommand{\href}[2]{#2}


\end{document}